\documentclass[a4paper,11pt]{article}
\usepackage[T1]{fontenc}
\usepackage[utf8]{inputenc}
\usepackage{fullpage}

\usepackage[all=normal,bibliography=tight]{savetrees}
\usepackage{cite}
\usepackage{amsmath,amssymb,amsfonts,amsthm}
\usepackage[tight]{subfigure}
\usepackage{tikz}
\usepackage{todonotes}
\usepackage[absolute]{textpos}
\usepackage{hyperref}
\usepackage[noabbrev,capitalize]{cleveref}
\usepackage{makecell}
\usepackage{enumitem} 
\usepackage{multirow}
\usepackage{xr} %
\usepackage{booktabs}

\usetikzlibrary{shapes,snakes}

\newtheorem{theorem}{Theorem}

\newtheorem{claim}{Claim}[theorem]
\theoremstyle{definition}

\newtheorem*{conjecture*}{Wegner's Conjecture}

\usepackage{etoolbox}

\newcommand{\wei}{{\mathfrak{w}}}

\usetikzlibrary{calc}
\newcommand{\rightstar}[2]
{
	\draw[gray] ($#1+(-30:0.6)$) -- #1 -- ($#1+(30:0.6)$);
	\node[gray] at ($#1+(0.5,0.1)$) {$\vdots$};
}
\newcommand{\leftstar}[2]
{
	\draw[gray] ($#1+(150:0.6)$) -- #1 -- ($#1+(210:0.6)$);
	\node[gray] at ($#1+(-0.5,0.1)$) {$\vdots$};
}
\newcommand{\upstar}[2]
{
	\draw[gray] ($#1+(120:0.6)$) -- #1 -- ($#1+(60:0.6)$);
	\node[gray] at ($#1+(0,0.5)$) {$\hdots$};
}
\newcommand{\downstar}[2]
{
	\draw[gray] ($#1+(-120:0.6)$) -- #1 -- ($#1+(-60:0.6)$);
	\node[gray] at ($#1+(0,-0.5)$) {$\hdots$};
}

\makeatletter
\newcommand\footnoteref[1]{\protected@xdef\@thefnmark{\ref{#1}}\@footnotemark}
\makeatother

\begin{document}
\title{Coloring squares of planar graphs with small maximum degree%
\footnote{
MK and ST were supported by the ``Szkoła Orłów'' (``School of Eagles'') project, co-financed by the European Social Fund under the Knowledge-Education-Development Operational Programme, Axis III, Higher Education For The Economy And Development, measure 3.1, Competences In Higher Education.
PRz was supported by the project that has received funding from the European Research Council (ERC) under the European Union's Horizon 2020 research and innovation programme Grant Agreement 714704.
}
}

\author{Mateusz Krzyzi\'{n}ski\footnote{Faculty of Mathematics and Information Science, Warsaw University of Technology, Warsaw, Poland, \texttt{m.krzyzinski@student.mini.pw.edu.pl}}     
    \and Pawe\l{} Rz\k{a}\.{z}ewski\footnote{Faculty of Mathematics and Information Science, Warsaw University of Technology, Warsaw, Poland\newline and Faculty of Mathematics, Informatics and Mechanics, University of Warsaw, Poland, \texttt{p.rzazewski@mini.pw.edu.pl}}
    \and Szymon Tur\footnote{Faculty of Mathematics and Information Science, Warsaw University of Technology, Warsaw, Poland, \texttt{s.tur@student.mini.pw.edu.pl}}     
}
\date{}

\maketitle

\begin{abstract}
For a graph $G$, by $\chi_2(G)$ we denote the minimum integer $k$, such that there is a $k$-coloring of the vertices of $G$ in which vertices at distance at most 2 receive distinct colors. Equivalently, $\chi_2(G)$ is the chromatic number of the square of $G$.
In 1977 Wegner conjectured that if $G$ is planar and has maximum degree $\Delta$,
then $\chi_2(G) \leq 7$ if $\Delta \leq 3$, 
$\chi_2(G) \leq \Delta+5$ if $4 \leq \Delta \leq 7$,
and $\lfloor 3\Delta/2 \rfloor +1$ if $\Delta \geq 8$.
Despite extensive work, the known upper bounds are quite far from the conjectured ones, especially for small values of $\Delta$.
In this work we show that for every planar graph $G$ with maximum degree $\Delta$ it holds that $\chi_2(G) \leq 3\Delta+4$.
This result provides the best known upper bound for $6 \leq \Delta \leq 14$.

\end{abstract}

\begin{textblock}{20}(0, 11.5)
\includegraphics[width=40px]{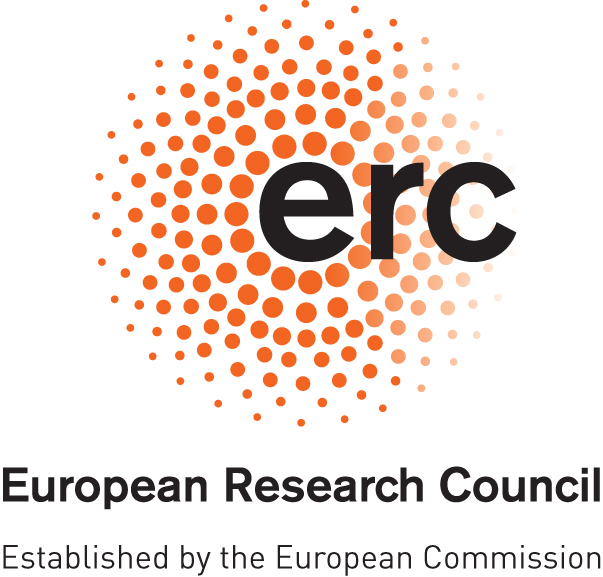}%
\end{textblock}
\begin{textblock}{20}(-0.25, 11.9)
\includegraphics[width=60px]{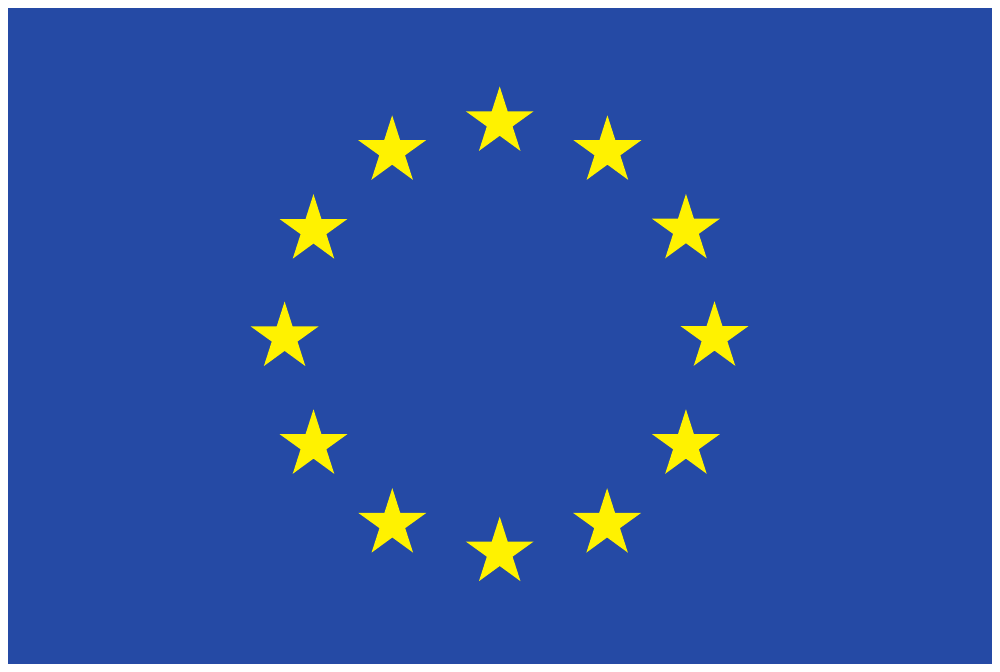}%
\end{textblock}

\section{Introduction}
Graph coloring is undoubtedly among the best studied problems in graph theory.
The origins of the research on graph colorings date back to 19th century and are related to the question whether all planar graphs can be properly colored with four color. The affirmative answer to this question, i.e., the celebrated four color theorem, remains one of the most famous results in graph theory~\cite{zbMATH03601602,zbMATH03601603,DBLP:journals/jct/RobertsonSST97}.
The study of the restrictions and generalizations of the problem~\cite{zbMATH03183649,DBLP:journals/jct/Thomassen94a,Alon1992,DBLP:journals/jctb/GrytczukZ20,DVORAK2020,DBLP:journals/jct/DvorakP18} led to many exciting results and much better understanding of the structure of planar graphs.

Other variants or coloring planar graphs were also considered. Already in 1977, Wegner~\cite{Wegner_conjecture} studied the problem of coloring graphs in such a way that the vertices with the same color must be at distance more than $d$, where $d$ is a fixed integer.
Such a coloring is called \emph{distance-$d$ coloring} and the minimum number of colors in a distance-$d$ coloring of a graph $G$ is denoted by $\chi_d(G)$. Note that for $d=1$ we obtain the classic graph coloring problem.
The next case that has received the most attention is $d=2$. The problem of distance-2 coloring is also known as \emph{$L(1,1)$-labeling}~\cite{DBLP:journals/cj/Calamoneri11}. Let us give a brief overview of the known results on distance-2 coloring of planar graphs.
In what follows $G$ is a planar graph with maximum degree $\Delta$.

First, observe that, in contrast to the classic coloring, there is no universal constant $c$ such that $\chi_2(G) \leq c$ for all planar graphs $G$.
Indeed, for the $n$-vertex star $K_{1,n-1}$ it holds that $\chi_2(K_{1,n-1}) = n$.
This implies that every graph $G$ satisfies $\chi_2(G) \geq \Delta+1$.
On the other hand, as every planar graph has a vertex of degree at most 5, a simple greedy algorithm yields the bound $\chi_2(G) \leq 5\Delta +1$. Thus $\chi_2(G)$ is bounded by a linear function of $\Delta$.

Wegner~\cite{Wegner_conjecture} was probably the first who studied this dependence.
Among other results, he showed that if $G$ is planar and has maximum degree at most 3, then $\chi_2(G) \leq 8$.
He also presented some families of planar graphs which require a large number of colors in any distance-2 coloring and conjectured that the lower bounds given by these families are actually tight.
This problem is known as Wegner's conjecture.

\begin{conjecture*}
Every planar graph $G$ with maximum degree $\Delta$ satisfies
\[
\chi_2(G) \leq
\begin{cases}
7 & \text{ if } \Delta \leq 3,\\
\Delta+5 & \text{ if } 4 \leq \Delta \leq 7,\\
\left\lfloor \frac{3\Delta}{2} \right\rfloor+1 & \text{ if } \Delta \geq 8.
\end{cases}
\]
\end{conjecture*}

The problem of bounding $\chi_2(G)$ received a considerable attention. However, despite many partial results, the only case for which we know tight bound is $\Delta=3$: Thomassen~\cite{Thomassen} confirmed the conjecture by showing that seven colors always suffice.
For $\Delta \geq 4$ the conjecture is wide open; we summarize the known bounds in \cref{tab:other_results}.

\begin{table}[h!]
\centering
\renewcommand{\arraystretch}{1.5}
\begin{tabular}{lll}
\toprule
\textbf{Authors}&     \textbf{Restriction}     &      \textbf{Result}  \\
\midrule \midrule
Thomassen~\cite{Thomassen} & $\Delta \leq 3$ & $\chi_2(G) \leq 7$ \\ \hline
Jonas \cite{Jonas}		&    $\Delta \geq 7$ 		    &     $\chi_2(G) \leq 8\Delta - 22$  \\\hline
Wong \cite{Wong} &    $\Delta \geq 7$ 		    &     $\chi_2(G) \leq 3\Delta + 5$ \\\hline
Madaras and Marcinova \cite{Madaras} &  $\Delta \geq 12$ &  $\chi_2(G) \leq 2\Delta + 18$ \\\hline
\multirow{3}{*}{Borodin et al.  \cite{Borodin}} &   $\Delta \leq 20$ &  $\chi_2(G) \leq 59$ \\
              			&  $21 \leq \Delta\leq 46$	&    $\chi_2(G) \leq \Delta + 39$ \\
          				&  $\Delta \geq 47$ 		   	&    $\chi_2(G) \leq \left\lceil \frac{9\Delta}{5} \right\rceil + 1$ \\\hline
\multirow{2}{*}{van den Heuvel and McGuinness  \cite{Heuvel} }&  $\Delta \geq 5$ &  $\chi_2(G) \leq 9\Delta - 19$ \\
						&         								&     $\chi_2(G) \leq 2\Delta + 25$ \\\hline
Agnarsson and Halldorsson \cite{Agnarsson} &  $\Delta \geq 749$ &     $\chi_2(G) \leq \left\lfloor \frac{9\Delta}{5} \right\rfloor + 2$ \\\hline
\multirow{2}{*}{Molloy and Salavatipour \cite{Molloy} } & $\Delta \geq 241$ &  $\chi_2(G) \leq \left\lceil \frac{5\Delta}{3} \right\rceil + 25$ \\
						&                           				&    $\chi_2(G) \leq \left\lceil \frac{5\Delta}{3} \right\rceil + 78$ \\\hline
\multirow{2}{*}{Zhu and Bu  \cite{J_Zhu} }& $\Delta \leq 5$ &    $\chi_2(G) \leq 20$ \\
  						&        $\Delta \geq 6$ 		&    $\chi_2(G) \leq 5\Delta - 7$ \\
\bottomrule
\end{tabular}
\caption{The progress on Wegner's conjecture.}\label{tab:other_results}
\end{table}

Let us highlight that the currently best known bound is $\left \lceil \frac{5\Delta}{3}   \right \rceil  + \mathcal{O}(1)$ by Molloy and Salavatipour~\cite{Molloy}. However, since the additive constant is large, this bound is very far from the conjectured one for small values of $\Delta$. Thus some attention has been put on refining the bounds for graphs with small maximum degree, see~\cref{fig:results_plot}.
We continue this line of research and show the following result.

\begin{theorem}
\label{main_theorem}
Every planar graph $G$ with maximum degree $\Delta$ satisfies $\chi_2(G) \leq 3\Delta+4$.
\end{theorem}

We point out that \cref{main_theorem} provides the best known upper bound for the cases $6 \leq \Delta \leq 14$.

\begin{figure}[h!]
\begin{center}
\resizebox{\textwidth}{!}{\input{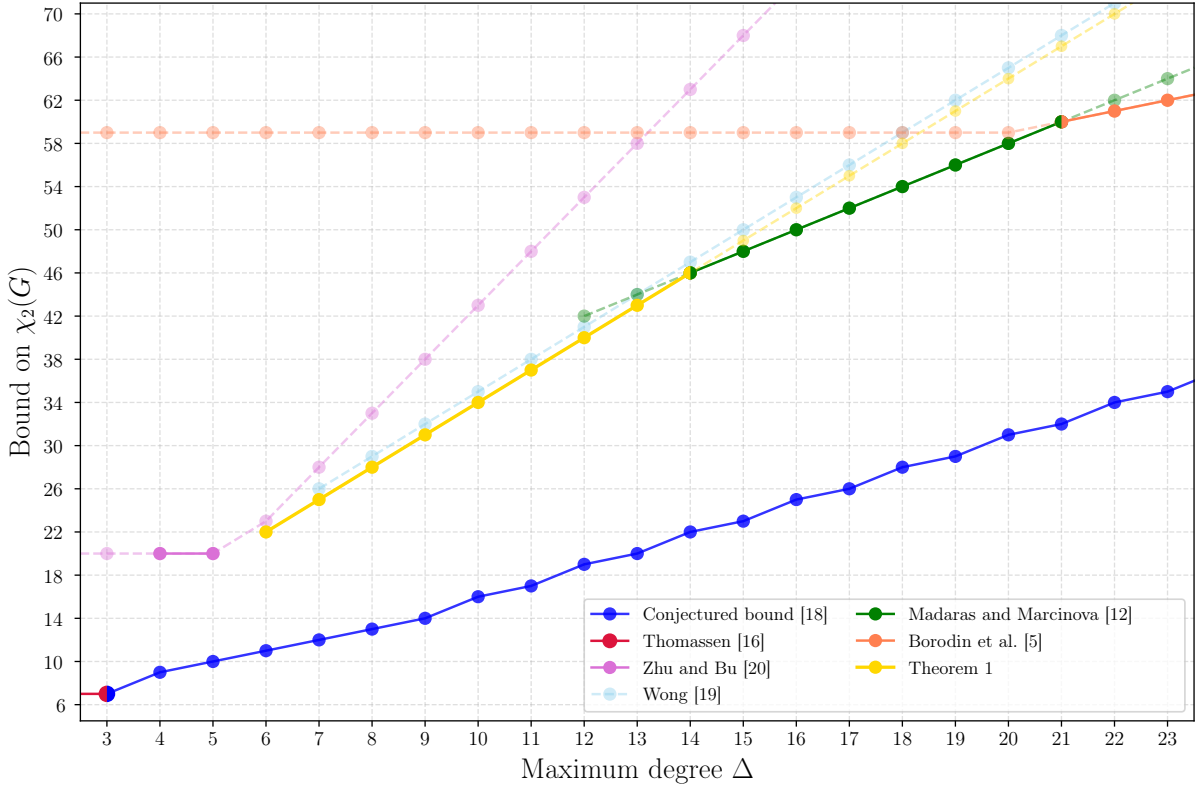}}
\end{center}
\caption{Dependence of $\chi_2(G)$ on $\Delta$ for $3 \leq \Delta \leq 23$.}
\label{fig:results_plot}
\end{figure}

The proof of \cref{main_theorem} uses the discharging method. We consider a minimal counterexample $G$ and we fix its plane embedding.
In \cref{sec:configurations} we show that $G$ cannot contains certain subgraphs, as this would contradict the minimality of $G$.
Then, in \cref{sec:discharging}, we distribute some integer values, called \emph{charges}, to the vertices and faces of $G$ in a way that the total charge is negative.
Next, we apply six \emph{discharging rules} to transfer charges between the vertices and faces of $G$. 
Eventually, we analyze the final charges and find out that every vertex and every face has nonnegative charge.
As in the discharging phase no charge is created nor lost, this is a contradiction.
Thus the counterexample to \cref{main_theorem} cannot exist.

\section{Preliminaries}
All graphs considered in the paper are simple and finite.
For a graph $G$, by $V(G)$ and $E(G)$ we denote, respectively, the vertex set and the edge set of $G$.
Furthermore, if $G$ is planar and given along with a fixed plane embedding, then $F(G)$ denotes the set of faces of $G$.

For two vertices $v$ and $u$, by $\mathrm{dist}_G(v, u)$ we denote the \emph{distance} between these vertices, i.e., the  number of edges on a shortest $u$-$v$ path in $G$. 
For a vertex $v$, by $N_G(v)$ we denote its \emph{neighborhood}, i.e., the set of all vertices adjacent to $v$, and by $\deg_G(v)$ we denote the \emph{degree} of a vertex $v$, i.e., $\left|N_G(v)\right|$. The maximum and the minimum degree of $G$ are denoted by, respectively, $\Delta(G)$ and $\delta(G)$.

For a vertex $v \in V(G)$, by $G-v$ we denote the graph obtained from $G$ by removing $v$ with all incident edges.
For $u,v \in V(G)$, by $G + uv$ we denote the graph with vertex set $V(G)$ and edge set $E(G) \cup \{uv\}$. Note that if $uv \in E(G)$, then $G = G + uv$.
In other words, we never create multiple edges.

Each face $f$ is bounded by a closed walk called a \emph{boundary}. We write $f = [v_1, v_2, \ldots, v_k]$ to denote the cyclic ordering of vertices along the the boundary of $f$. Note that if $G$ is 2-connected, then the boundary of each face is a cycle; we will always work in this setting.

We say that a vertex $v$ and a face $f$ are \emph{incident} if $v$ lies on the boundary of $f$.
By $\deg_G(f)$ we denote the \emph{degree} of a face $f$, i.e., the number of vertices incident to $f$.

If the graph $G$ is clear from the context, we drop the subscript in the notation above. 

For an integer $d$, a vertex $v$ is said to be a $d$-vertex (respectively, a $d^+$-vertex, a $d^-$-vertex) if its degree is exactly $d$ (respectively, at least $d$, at most $d$).
Similarly, a face $f$ is said to be a $d$-face (respectively,  a $d^+$-face, a $d^-$-face)  if its degree is exactly $d$ (respectively, at least $d$, at most $d$).

Often we will consider a situation where some subset of vertices of a graph $G$ is colored.
For an uncolored vertex $v$, we say that a color is \emph{blocked} if it appears on a vertex within distance at most 2 from $v$.
A color that is not blocked is \emph{free}.

\section{Main proof}

For contradiction suppose that \cref{main_theorem} does not hold and let $G$ be a minimum counterexample.
Thus, for any planar graph $G'$, if $|V(G')| + |E(G')| < |V(G)| + |E(G)|$ and $\Delta(G') \leq \Delta$, then $\chi_2(G') \leq 3\Delta +4$.

Observe that by the already mentioned result by Zhu and Bu~\cite{J_Zhu} we can safely assume that $\Delta \geq 6$. Furthermore, we can assume that $G$ is connected, as the coloring of $G$ can be obtained by coloring each connected component independently and each of them is smaller than $G$.

Fix some plane embedding of $G$. Whenever we refer to faces of $G$, we mean the faces of this fixed plane embedding.

\subsection{Forbidden configurations}\label{sec:configurations}
In this section we present a series of technical claims in which we analyze the structure of the graph $G$.

\begin{claim}\label{clm:2connected}
$G$ is 2-connected.
\end{claim}
\begin{proof}
Assume that $G$ has a cutvertex $v$ and let $C$ be the vertex set of one connected component of $G - v$.
Define $G'_1 = G[C \cup \{v\}]$ and $G'_2 := G - C$. By the minimality of $G$, for each $i \in \{1,2\}$ there is a distance-2 $(3\Delta + 4)$-coloring $\varphi_i$ of $G'_i$.
We can permute the colors in $\varphi_2$ so that (i)~$\varphi_1(v) = \varphi_2(v)$ and (ii) $\varphi_1(N_{G_1}(v)) \cap \varphi_2(N_{G_2}(v)) = \emptyset$. The union of these colorings is a distance-2  $(3\Delta + 4)$-coloring of $G$, a contradiction.
\end{proof}

So by \cref{clm:2connected} from now on we can assume that the boundary of each face is a simple cycle with at least three edges.

The proofs of the next few claims follow the same outline.
First, assume that $G$ contains some configuration that we want to exclude.
We modify $G$ by removing a single vertex $v$ and possibly adding some new edges,
in order to obtain a graph $G'$ with the following properties:
\begin{enumerate}[label=(\roman*)]
\item $|V(G')| + |E(G')| \leq |V(G)| + |E(G)|$,
\item $\Delta(G') \leq \Delta$,
\item $G'$ is planar (and its plane embedding can be easily obtained from the plane embedding of $G$),
\item all pairs of vertices in $V(G) \setminus \{v\}$ that are at distance at most 2 in $G$ are at distance at most 2 in $G'$.
\end{enumerate}
By properties (i), (ii), and (iii) and the minimality of $G$ we observe that $G'$ has a distance-2 $(3\Delta + 4)$-coloring $\varphi$.
By property (iv), we can safely color all vertices of $V(G) \setminus \{v\}$ according to $v$. To obtain a distance-2 $(3\Delta + 4)$-coloring for $G$ we only need to find a color for $v$. We do this by ensuring that the number of colors that are blocked for $v$ is strictly less than $3\Delta+4$. Thus there is a free color for $v$, as $G$ admits a distance-2 $(3\Delta + 4)$-coloring, a contradiction.

For brevity, in the proofs we only say how to define $G'$ and compute the number of colors that are blocked for $v$.
In particular, we will not explicitly check properties (i)--(iv), as verifying them is straightforward.

\begin{claim}\label{delta3}
$\delta(G) \geq 3$.
\end{claim}
\begin{proof}
First suppose that $\deg(v)=1$.
We set $G' := G - v$ and observe that a distance-2 $(3\Delta + 4)$-coloring of $G'$ blocks at most $\Delta < 3\Delta+4$ colors for $v$.

Now, assume that $\deg(v)=2$, let $N(v) = \{v_1, v_2\}$, see \cref{fig:delta3}.
Let $G' := G - v + v_1v_2$.
We observe that a distance-2 $(3\Delta + 4)$-coloring of $G'$ blocks at most $2\Delta < 3\Delta+4$ colors for $v$.
\end{proof}
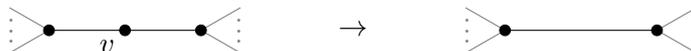
\begin{figure}[h!]
	\begin{center}	
		\begin{tikzpicture}
			\coordinate (a) at (0,0);
			\coordinate (b) at (1,0);
			\coordinate (c) at (-1,0);
			\draw (c) -- (a) -- (b);
			\rightstar{(b)}{\Delta-1}
			\leftstar{(c)}{\Delta-1}
			\filldraw[black] (a) circle (2pt) node[anchor=north east] {$v$};
			\filldraw[black] (b) circle (2pt);
			\filldraw[black] (c) circle (2pt);
	
			\node at (3,0) {$\to$};
			\coordinate (d) at (7,0);
			\coordinate (e) at (5,0);
			\rightstar{(d)}{\Delta-1}
			\leftstar{(e)}{\Delta-1}
			
			\draw (d) -- (e);
	
			\filldraw[black] (d) circle (2pt);
			\filldraw[black] (e) circle (2pt);
		\end{tikzpicture}
	\end{center}
\caption{Forbidden configuration in \cref{delta3}.}\label{fig:delta3}
\end{figure}

\begin{claim}\label{adj_3v}
$G$ has no two adjacent 3-vertices.
\end{claim}
\begin{proof}
Suppose $u$ and $v$ are adjacent 3-vertices and let $N(v) = \{u,v_1,v_2\}$, see \cref{fig:adj_3v}.
Let $G' = G - v + uv_1 + uv_2$.  
We observe that a distance-2 $(3\Delta + 4)$-coloring of $G'$ blocks at most $2\Delta+3 < 3\Delta+4$ colors for $v$.
\end{proof}
\begin{figure}[h!]	
	\begin{center}
		\begin{tikzpicture}
		\coordinate (u) at (1,0);
		\coordinate (v) at (0,0);
		\coordinate (v1) at ($(v)+(140:1)$);
		\coordinate (v2) at ($(v)+(220:1)$);
		
		\coordinate (u1) at ($(u)+(40:0.6)$);
		\coordinate (u2) at ($(u)+(-40:0.6)$);
		
			\leftstar{(v1)}{\Delta-1}
			\leftstar{(v2)}{\Delta-1}
		
		\draw (v1) -- (v) -- (v2);
		\draw (u) -- (v);
		
		\draw[gray] (u1) -- (u) -- (u2);

		\filldraw[black] (u) circle (2pt) node[anchor=north east] {$u$};
		\filldraw[black] (v) circle (2pt) node[anchor=north west] {$v$};
		\filldraw[black] (v1) circle (2pt) node[anchor=south west] {$v_1$};
		\filldraw[black] (v2) circle (2pt) node[anchor=north west] {$v_2$};
		
		\node at (3,0) {$\to$};
		
		\coordinate (up) at (7,0);
		\coordinate (vp) at (6,0);
		\coordinate (vp1) at ($(vp)+(140:1)$);
		\coordinate (vp2) at ($(vp)+(220:1)$);
		
		\coordinate (up1) at ($(up)+(40:0.6)$);
		\coordinate (up2) at ($(up)+(-40:0.6)$);
		
			\leftstar{(vp1)}{\Delta-1}
			\leftstar{(vp2)}{\Delta-1}
		
		\draw (vp1) -- (up) -- (vp2);

		\draw[gray] (up1) -- (up) -- (up2);

		\filldraw[black] (up) circle (2pt) node[anchor=north] {$u$};
		
		\filldraw[black] (vp1) circle (2pt) node[anchor=south west] {$v_1$};
		\filldraw[black] (vp2) circle (2pt) node[anchor=north west] {$v_2$};
		\end{tikzpicture}
	\end{center}
	\caption{Forbidden configuration in \cref{adj_3v}.}	\label{fig:adj_3v}
\end{figure}
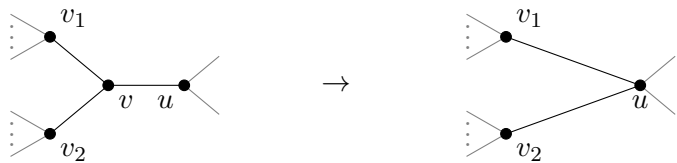

\begin{claim}\label{3v_inc_3f}
$G$ has no 3-vertex incident to a 3-face.
\end{claim}
\begin{proof}
Suppose that $G$ has a 3-vertex $v$ incident to a 3-face  $[v,v_1,v_2]$ and let $N(v) = \{v_1, v_2, v_3\}$, see \cref{fig:3v_inc_3f}.
Let $G' = G - v + v_1v_3$.
We observe that a distance-2 $(3\Delta + 4)$-coloring of $G'$ blocks at most $\Delta + 2 (\Delta - 1) = 3\Delta - 2 < 3\Delta+4$ colors for $v$.
\end{proof}
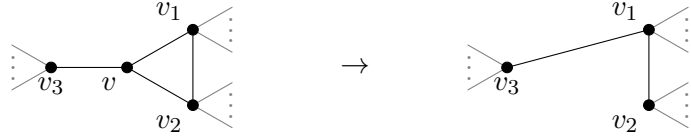
\begin{figure}[h!]
	\begin{center}
	\begin{tikzpicture}
	\coordinate (v) at (0,0);
	\coordinate (v3) at (-1,0);
	\coordinate (v1) at ($(v)+(30:1)$);
	\coordinate (v2) at ($(v)+(-30:1)$);
	\draw (v1) -- (v) -- (v2) -- (v1);
	\draw (v3) -- (v);
	\rightstar{(v1)}{\Delta-2}
	\rightstar{(v2)}{\Delta-2}
	\leftstar{(v3)}{\Delta-1}
	
	\filldraw[black] (v3) circle (2pt) node[anchor=north] {$v_3$};
	\filldraw[black] (v) circle (2pt) node[anchor=north east] {$v$};
	\filldraw[black] (v1) circle (2pt) node[anchor=south east] {$v_1$};
	\filldraw[black] (v2) circle (2pt) node[anchor=north east] {$v_2$};
	
	\node at (3,0) {$\to$};
	
	\coordinate (vp) at (6,0);
	\coordinate (vp3) at (5,0);
	\coordinate (vp1) at ($(vp)+(30:1)$);
	\coordinate (vp2) at ($(vp)+(-30:1)$);
	\draw (vp3) -- (vp1) -- (vp2);
	
	\rightstar{(vp1)}{\Delta-2}
	\rightstar{(vp2)}{\Delta-2}
	\leftstar{(vp3)}{\Delta-1}
	
	\filldraw[black] (vp3) circle (2pt) node[anchor=north] {$v_3$};
	\filldraw[black] (vp1) circle (2pt) node[anchor=south east] {$v_1$};
	\filldraw[black] (vp2) circle (2pt) node[anchor=north east] {$v_2$};
	\end{tikzpicture}
	\end{center}
	\caption{Forbidden configuration in \cref{3v_inc_3f}.}
	\label{fig:3v_inc_3f}
\end{figure}

\begin{claim}\label{3v_inc_2_4f}
$G$ has no 3-vertex incident to two 4-faces.
\end{claim}
\begin{proof}
Suppose that $G$ has a 3-vertex $v$ incident to two 4-faces:  $[v,v_1,u,v_2]$ and $[v,v_3,w,v_2]$, see \cref{fig:3v_inc_2_4f}.
Let $G' = G - v + v_1v_3$. 
We observe that a distance-2 $(3\Delta + 4)$-coloring of $G'$ blocks at most $2 (\Delta - 1) + \Delta = 3\Delta - 2 < 3\Delta+4$ colors for $v$.
\end{proof}
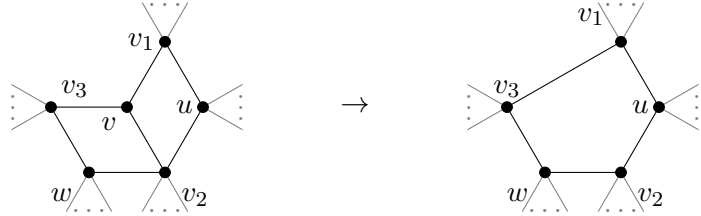
\begin{figure}[h!]
	\begin{center}
		\begin{tikzpicture}
		\coordinate (v) at (0,0);
		\coordinate (v1) at ($(v)+(60:1)$);
		\coordinate (v2) at ($(v)+(-60:1)$);
		\coordinate (v3) at ($(v)+(-1,0)$);
		
		\coordinate (u) at ($(v)+(1,0)$);
		\coordinate (w) at ($(v)+(-120:1)$);
		
		\draw (v1) -- (v) -- (v2) -- (u) -- (v1);
		\draw (v) -- (v3) -- (w) -- (v2); 
		
		\upstar{(v1)}{\Delta-2}
		\leftstar{(v3)}{\Delta-2}
		\rightstar{(u)}{\Delta-2}
		\downstar{(v2)}{\Delta-3}
		\downstar{(w)}{\Delta-2}
		
		\filldraw[black] (v) circle (2pt) node[anchor=north east] {$v$};
		\filldraw[black] (v1) circle (2pt) node[anchor=east] {$v_1$};
		\filldraw[black] (v2) circle (2pt) node[inner sep=6pt,anchor=north west] {$v_2$};
		\filldraw[black] (v3) circle (2pt) node[anchor=south west] {$v_3$};
		\filldraw[black] (u) circle (2pt) node[anchor=east] {$u$};
		\filldraw[black] (w) circle (2pt) node[inner sep=6pt,anchor=north east] {$w$};
		
		\node at (3,0) {$\to$};
		
		\coordinate (vp) at (6,0);
		\coordinate (vp1) at ($(vp)+(60:1)$);
		\coordinate (vp2) at ($(vp)+(-60:1)$);
		\coordinate (vp3) at ($(vp)+(-1,0)$);
		
		\coordinate (up) at ($(vp)+(1,0)$);
		\coordinate (wp) at ($(vp)+(-120:1)$);
		
		\draw (vp2) -- (up) -- (vp1) -- (vp3) -- (wp) -- (vp2); 
		
		\upstar{(vp1)}{\Delta-2}
		\leftstar{(vp3)}{\Delta-2}
		\rightstar{(up)}{\Delta-2}
		\downstar{(vp2)}{\Delta-3}
		\downstar{(wp)}{\Delta-2}

		\filldraw[black] (vp1) circle (2pt) node[inner sep=6pt,anchor=south east] {$v_1$};
		\filldraw[black] (vp2) circle (2pt) node[inner sep=6pt,anchor=north west] {$v_2$};
		\filldraw[black] (vp3) circle (2pt) node[anchor=south] {$v_3$};
		\filldraw[black] (up) circle (2pt) node[anchor=east] {$u$};
		\filldraw[black] (wp) circle (2pt) node[inner sep=6pt,anchor=north east] {$w$};
		
		\end{tikzpicture}
	\end{center}
		\caption{Forbidden configuration in \cref{3v_inc_2_4f}.}\label{fig:3v_inc_2_4f}
\end{figure}

\begin{claim}\label{4v_inc_3f_466}
If a 4-vertex of $G$ is incident to a 3-face,
then the other two vertices on that 3-face are $6^+$-vertices.
\end{claim}
\begin{proof}
Suppose that $G$ has a 4-vertex $v$ incident to a 3-face $[v,v_1,v_2]$, where $\deg(v_1) \leq 5$ and $N(v) = \{v_1, v_2, v_3, v_4\}$, see \cref{fig:4v_inc_3f_466}.
Let $G' = G - v + v_1v_3 + v_1v_4$.
We observe that a distance-2 $(3\Delta + 4)$-coloring of $G'$ blocks at most $2 \Delta + (\Delta - 1) + 4= 3\Delta + 3 < 3\Delta+4$ colors for $v$.
\end{proof}
\begin{figure}[h!]
	\begin{center}
		\begin{tikzpicture}
		\coordinate (v) at (0,0);
		\coordinate (v1) at (-1,0);
		\coordinate (v2) at (0,-1);
		\coordinate (v3) at (1,0);
		\coordinate (v4) at (0,1);
		
		\draw (v) -- (v1) -- (v2) -- (v) -- (v3);
		\draw (v) -- (v4);
		
		\upstar{(v4)}{\Delta-1}
		\downstar{(v2)}{\Delta-2}
		\rightstar{(v3)}{\Delta-1}
		
		\draw[gray] (v1) -- ($(v1)+(150:0.6)$);
		\draw[gray] ($(v1)+(-0.6,0)$) -- (v1) -- ($(v1)+(210:0.6)$);
		
		\filldraw[black] (v) circle (2pt) node[anchor=north east] {$v$};
		\filldraw[black] (v1) circle (2pt) node[anchor=north] {$v_1$};
		\filldraw[black] (v2) circle (2pt) node[inner sep=6pt,anchor=north west] {$v_2$};
		\filldraw[black] (v3) circle (2pt) node[anchor=north east] {$v_3$};
		\filldraw[black] (v4) circle (2pt) node[anchor=north east] {$v_4$};
		
		\node at (3,0) {$\to$};
		
		\coordinate (v) at (6,0);
		\coordinate (v1) at (5,0);
		\coordinate (v2) at (6,-1);
		\coordinate (v3) at (7,0);
		\coordinate (v4) at (6,1);
		
		\draw (v1) -- (v2);
		\draw (v3) -- (v1) -- (v4);
		
		\upstar{(v4)}{\Delta-1}
		\downstar{(v2)}{\Delta-2}
		\rightstar{(v3)}{\Delta-1}
		
		\draw[gray] (v1) -- ($(v1)+(150:0.6)$);
		\draw[gray] ($(v1)+(-0.6,0)$) -- (v1) -- ($(v1)+(210:0.6)$);

		\filldraw[black] (v1) circle (2pt) node[anchor=north] {$v_1$};
		\filldraw[black] (v2) circle (2pt) node[inner sep=6pt,anchor=north west] {$v_2$};
		\filldraw[black] (v3) circle (2pt) node[anchor=north east] {$v_3$};
		\filldraw[black] (v4) circle (2pt) node[anchor=north west] {$v_4$};
		\end{tikzpicture}
	\end{center}
	\caption{Forbidden configuration in \cref{4v_inc_3f_466}.}	\label{fig:4v_inc_3f_466}
\end{figure}
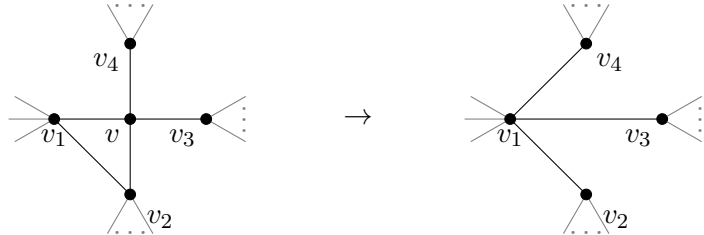

\begin{claim}\label{4v_2_3f_8v}
If $G$ has two 3-faces that share an edge $uv$ on their boundaries and $\deg(u)=4$, then $\deg(v) \geq 8$.
\end{claim}
\begin{proof}
Suppose that $G$ has a 4-vertex $u$ incident to two adjacent 3-faces: $[v,v_1,u]$ and $[v,v_2,u]$, see \cref{fig:4v_2_3f_8v}.
where $\deg(v) \leq 7$ and $N(u) = \{u, v_1, v_2, v_3\}$. Let $G' = G - u + vv_3$. 
We observe that a distance-2 $(3\Delta + 4)$-coloring of $G'$ blocks at most $\Delta + 2 (\Delta - 1) + 5  = 3\Delta + 3 < 3\Delta+4$ colors for $u$.
\end{proof}
\begin{figure}[h!]
	\begin{center}
		\begin{tikzpicture}
		\coordinate (v) at (0,0);
		\coordinate (v1) at (-1,0);
		\coordinate (v2) at (0,-1);
		\coordinate (v3) at (1,0);
		\coordinate (v4) at (0,1);
		
		\draw (v) -- (v1) -- (v2) -- (v) -- (v3);
		\draw (v) -- (v4) -- (v1);
		
		\upstar{(v4)}{\Delta-2}
		\downstar{(v2)}{\Delta-2}
		\rightstar{(v3)}{\Delta-1}
		
		\draw[gray] ($(v1)+(190:0.6)$) -- (v1) -- ($(v1)+(150:0.6)$);
		\draw[gray] ($(v1)+(170:0.6)$) -- (v1) -- ($(v1)+(210:0.6)$);
		
		\filldraw[black] (v) circle (2pt) node[anchor=north east] {$u$};
		\filldraw[black] (v1) circle (2pt) node[anchor=north] {$v$};
		\filldraw[black] (v2) circle (2pt) node[inner sep=6pt,anchor=north west] {$v_2$};
		\filldraw[black] (v3) circle (2pt) node[anchor=north east] {$v_3$};
		\filldraw[black] (v4) circle (2pt) node[anchor=north west] {$v_1$};
		
		\node at (3,0) {$\to$};
		
		\coordinate (v) at (6,0);
		\coordinate (v1) at (5,0);
		\coordinate (v2) at (6,-1);
		\coordinate (v3) at (7,0);
		\coordinate (v4) at (6,1);
		
		\draw (v1) -- (v2);
		\draw (v3) -- (v1) -- (v4);
		
		\upstar{(v4)}{\Delta-1}
		\downstar{(v2)}{\Delta-2}
		\rightstar{(v3)}{\Delta-1}
		
		\draw[gray] ($(v1)+(190:0.6)$) -- (v1) -- ($(v1)+(150:0.6)$);
		\draw[gray] ($(v1)+(170:0.6)$) -- (v1) -- ($(v1)+(210:0.6)$);

		\filldraw[black] (v1) circle (2pt) node[anchor=north] {$v$};
		\filldraw[black] (v2) circle (2pt) node[inner sep=6pt,anchor=north west] {$v_2$};
		\filldraw[black] (v3) circle (2pt) node[anchor=north east] {$v_3$};
		\filldraw[black] (v4) circle (2pt) node[anchor=north west] {$v_1$};
		\end{tikzpicture}
	\end{center}
\caption{Forbidden configuration in \cref{4v_2_3f_8v}.}	\label{fig:4v_2_3f_8v}
\end{figure}
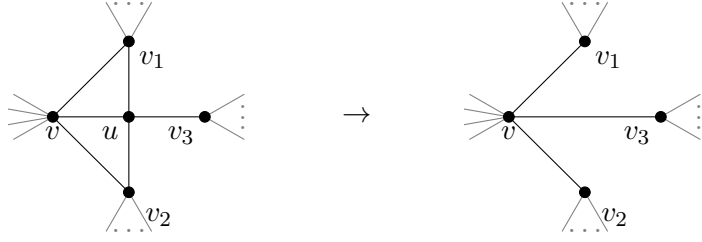

\cref{4v_2_3f_8v} implies two statements that will be directly used in our proof. Let us start with some definitions.
A 3-face is \emph{weird} if it is incident to a 4-vertex and two $6^+$-vertices.
Consider a vertex $v$ and let $f_0,f_1,\ldots,f_{d-1}$ be the sequence of faces incident to $v$ in the cyclic ordering around $v$ in our fixed plane embedding of $G$. Arithmetic  operations on indices will be performed modulo $d$.
A \emph{fan centered at $v$} is a sequence $\mathbf{f} = f_i,f_{i+1},\ldots,f_{i+p}$, such that (i) each element of $\mathbf{f}$ is a 3-face, and (ii) $f_{i-1}$ and $f_{i+p+1}$ are not 3-faces (it is possible that $f_{i-1}=f_{i+p+1}$); see \cref{fig:fan}.
The faces $f_i$ and $f_{i+p}$ are \emph{outer faces} of $\mathbf{f}$.
By the vertices of $\mathbf{f}$ we mean the vertices incident to the faces of $\mathbf{f}$, except for $v$.
The vertices of $f_i$ and $f_{i+p}$ that are not shared with other elements of $\mathbf{f}$ are \emph{outer vertices} of $\mathbf{f}$.
\begin{figure}[h!]
	\begin{center}	
		\begin{tikzpicture}
		\coordinate (v) at (0,0);
		\fill[gray] (v) -- ($(v)+(45*5:1.5)$) -- ($(v)+(45*6:1.5)$);
		\draw ($(v)+(45*5:1.5)$) -- ($(v)+(45*6:1.5)$);

		\leftstar{($(v)+(45*3:1.5)$)}{}
		\leftstar{($(v)+(45*4:1.5)$)}{}
		\leftstar{($(v)+(45*5:1.5)$)}{}

		\rightstar{($(v)+(45*7:1.5)$)}{}
		\rightstar{($(v)+(45*8:1.5)$)}{}
		\rightstar{($(v)+(45*1:1.5)$)}{}
		
		\upstar{($(v)+(45*2:1.5)$)}{}
		
		\downstar{($(v)+(45*6:1.5)$)}{}
		
		\foreach \x in {1,4}
		{
			
			\fill[gray] (v) -- ($(v)+(45*\x-45:1.5)$) -- ($(v)+(45*\x:1.5)$);
			\draw ($(v)+(45*\x-45:1.5)$) -- ($(v)+(45*\x:1.5)$);
		}
		\foreach \x in {2,3}
		{
			
			\fill[lightgray] (v) -- ($(v)+(45*\x-45:1.5)$) -- ($(v)+(45*\x:1.5)$);
			\draw ($(v)+(45*\x-45:1.5)$) -- ($(v)+(45*\x:1.5)$);
		}
		\foreach \x in {1,2,3,4,5,6,7,8}
		{
			\draw (v) -- ($(v)+(45*\x:1.5)$);
			\filldraw[black] ($(v)+(45*\x:1.5)$) circle (2pt);
		}
		\filldraw[black] ([xshift=-2pt,yshift=-2pt]$(v)+(45*8:1.5)$) rectangle ++(4pt,4pt);
		\filldraw[black] ([xshift=-2pt,yshift=-2pt]$(v)+(45*4:1.5)$) rectangle ++(4pt,4pt);
		\filldraw[black] ([xshift=-2pt,yshift=-2pt]$(v)+(45*5:1.5)$) rectangle ++(4pt,4pt);
		\filldraw[black] ([xshift=-2pt,yshift=-2pt]$(v)+(45*6:1.5)$) rectangle ++(4pt,4pt);
		\filldraw (v) circle (2pt) node[anchor=north west,inner xsep=9pt] {$v$};
		
		\end{tikzpicture}
	\end{center}
	\caption{Fans centered at a vertex $v$ are shaded and their outer faces are darker.
		Outer vertices are marked with squares.}\label{fig:fan}
\end{figure}
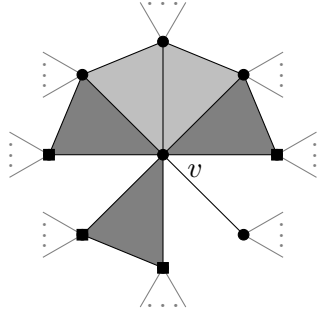

\begin{claim}\label{noweirdface}
If $v$ is a $7^-$ vertex and all faces incident to $v$ are 3-faces, then $v$ is not incident to a weird 3-face.
\end{claim}
\begin{proof}
By \cref{4v_2_3f_8v} we observe that $v$ is adjacent to no 4-vertex and thus is not incident to a weird 3-face.
\end{proof}

\begin{claim}\label{weirdouter}
Let $v$ be a $7^-$ vertex and let $\mathbf{f}$ be a fan centered at $v$.
If $\mathbf{f}$ contains a weird 3-face $f$, then $f$ is an outer face of $\mathbf{f}$ and the 4-vertex of $f$ is an outer vertex of $\mathbf{f}$.
\end{claim}
\begin{proof}
Observe that if the 4-vertex $v$ of $f$ is not an outer vertex of $\mathbf{f}$, then the edge $uv$ belongs to the boundaries of two 3-faces, which contradicts \cref{4v_2_3f_8v}.
\end{proof}

\begin{claim}\label{5v_inc_5_3f}
If a 5-vertex of $G$ is incident to five 3-faces, then it has at least four $7^+$-vertices as neighbors. 
\end{claim}
\begin{proof}
Suppose that $G$ has a 5-vertex $v$ incident to five 3-faces, see \cref{fig:5v_inc_5_3f}.
Let $v_1,v_2$ be distinct $6^-$-vertices in $N(v)$.
Let $G' = G - v$. 
We observe that a distance-2 $(3\Delta + 4)$-coloring of $G'$ blocks at most $2 \cdot 4 + 3 (\Delta - 2) = 3\Delta + 2 < 3\Delta+4$ colors for $v$.
\end{proof}
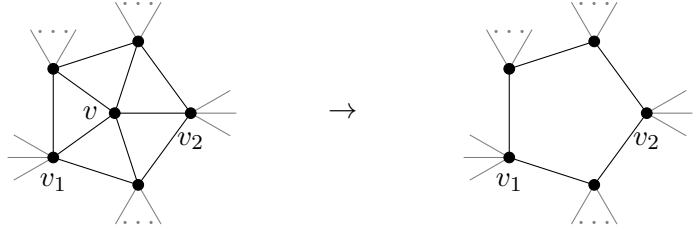
\begin{figure}[h!]
		\begin{center}
		\begin{tikzpicture}
		\coordinate (v) at (0,0);
		\foreach \x in {1,2,3,4,5}
		{
			\coordinate (v\x) at ($(v)+(72*\x:1)$);
			\draw (v) -- (v\x);
		}
		
		\upstar{(v1)}{\Delta-3}
		\upstar{(v2)}{\Delta-3}
		\downstar{(v4)}{\Delta-3}
		
		\draw (v1) -- (v2) -- (v3) -- (v4) -- (v5) -- (v1);
		
		\draw[gray] (v3) -- ($(v3)+(150:0.6)$);
		\draw[gray] ($(v3)+(-0.6,0)$) -- (v3) -- ($(v3)+(210:0.6)$);
		
		\draw[gray] (v5) -- ($(v5)+(30:0.6)$);
		\draw[gray] ($(v5)+(0.6,0)$) -- (v5) -- ($(v5)+(-30:0.6)$);
		
		\filldraw[black] (v) circle (2pt) node [inner sep=6pt,anchor=east] {$v$};
		\filldraw[black] (v1) circle (2pt);
		\filldraw[black] (v2) circle (2pt);
		\filldraw[black] (v3) circle (2pt) node [inner sep=6pt,anchor=north] {$v_1$};
		\filldraw[black] (v4) circle (2pt);
		\filldraw[black] (v5) circle (2pt) node [inner sep=7pt,anchor=north] {$v_2$};
		
		\node at (3,0) {$\to$};
		
		\coordinate (v) at (6,0);
		\foreach \x in {1,2,3,4,5}
		{
			\coordinate (v\x) at ($(v)+(72*\x:1)$);
			
		}
		
		\draw (v1) -- (v2) -- (v3) -- (v4) -- (v5) -- (v1);
		
		\upstar{(v1)}{\Delta-3}
		\upstar{(v2)}{\Delta-3}
		\downstar{(v4)}{\Delta-3}
		
		\draw[gray] (v3) -- ($(v3)+(150:0.6)$);
		\draw[gray] ($(v3)+(-0.6,0)$) -- (v3) -- ($(v3)+(210:0.6)$);
		
		\draw[gray] (v5) -- ($(v5)+(30:0.6)$);
		\draw[gray] ($(v5)+(0.6,0)$) -- (v5) -- ($(v5)+(-30:0.6)$);

		\filldraw[black] (v1) circle (2pt);
		\filldraw[black] (v2) circle (2pt);
		\filldraw[black] (v3) circle (2pt) node [inner sep=6pt,anchor=north] {$v_1$};
		\filldraw[black] (v4) circle (2pt);
		\filldraw[black] (v5) circle (2pt) node [inner sep=7pt,anchor=north] {$v_2$};
		\end{tikzpicture}
	\end{center}
	\caption{Forbidden configuration in \cref{5v_inc_5_3f}.}	\label{fig:5v_inc_5_3f}
\end{figure}

\begin{claim}\label{obs_5v_inc_4_3f}
Let $v$ be a 5-vertex of $G$ adjacent to exactly one $4^+$-face, and let $v_1$ and $v_2$ be the neighbors of $v$ on the boundary of this $4^+$-face.
Then either both $v_1$ and $v_2$ are $6^+$-vertices, or at least one of them is a $7^+$-vertex.
\end{claim}
\begin{proof}
Let $v,v_1,v_2$ be as in the assumptions of the claim (see \cref{fig:obs_5v_inc_4_3f}) and suppose that the statement does not hold.
In particular, $\deg(v_1) + \deg(v_2) \leq 11$.
Let $G' = G - v + v_1v_2$.  
We observe that a distance-2 $(3\Delta + 4)$-coloring of $G'$ blocks at most $ (\deg(v_1) + \deg(v_2) - 2) + 3 (\Delta - 2) \leq 3\Delta + 3 < 3\Delta+4$ colors for $v$.
\end{proof}
\begin{figure}[h!]	
	\begin{center}
		\begin{tikzpicture}
		\coordinate (v) at (0,0);
		\foreach \x in {1,2,3,4,5}
		{
			\coordinate (v\x) at ($(v)+(72*\x:1)$);
			\draw (v) -- (v\x);
		}
		
		\draw (v1) -- (v2);
		\draw (v3) -- (v4) -- (v5) -- (v1);
		
		\leftstar{(v2)}{d(v_2)-2}
		\leftstar{(v3)}{d(v_1)-2}
		\node[gray,anchor=east] at ($(v2)+(-0.5,0)$) {$\deg(v_2)\left\{\begin{matrix}{}\\{}\end{matrix}\right.$};
		\node[gray,anchor=east] at ($(v3)+(-0.5,0)$) {$\deg(v_1)\left\{\begin{matrix}{}\\{}\end{matrix}\right.$};
		
		\upstar{(v1)}{\Delta-3}
		\rightstar{(v5)}{\Delta-3}
		\downstar{(v4)}{\Delta-3}
		
		\filldraw[black] (v) circle (2pt) node [inner sep=6pt,anchor=east] {$v$};
		\filldraw[black] (v1) circle (2pt);
		\filldraw[black] (v2) circle (2pt) node [inner sep=6pt,anchor=south] {$v_2$};
		\filldraw[black] (v3) circle (2pt) node [inner sep=6pt,anchor=north] {$v_1$};
		\filldraw[black] (v4) circle (2pt);
		\filldraw[black] (v5) circle (2pt);
		
		\node at (3,0) {$\to$};
		
		\coordinate (v) at (7.5,0);
		\foreach \x in {1,2,3,4,5}
		{
			\coordinate (v\x) at ($(v)+(72*\x:1)$);
			
		}
		
		\draw (v1) -- (v2) -- (v3) -- (v4) -- (v5) -- (v1);
		
		\leftstar{(v2)}{d(v_2)-2}
		\leftstar{(v3)}{d(v_1)-2}
		\node[gray,anchor=east] at ($(v2)+(-0.5,0)$) {$\deg(v_2)\left\{\begin{matrix}{}\\{}\end{matrix}\right.$};
		\node[gray,anchor=east] at ($(v3)+(-0.5,0)$) {$\deg(v_1)\left\{\begin{matrix}{}\\{}\end{matrix}\right.$};
		
		\upstar{(v1)}{\Delta-3}
		\rightstar{(v5)}{\Delta-3}
		\downstar{(v4)}{\Delta-3}

		\filldraw[black] (v1) circle (2pt);
		\filldraw[black] (v2) circle (2pt) node [inner sep=6pt,anchor=south] {$v_2$};
		\filldraw[black] (v3) circle (2pt) node [inner sep=6pt,anchor=north] {$v_1$};
		\filldraw[black] (v4) circle (2pt);
		\filldraw[black] (v5) circle (2pt);
		\end{tikzpicture}
	\end{center}
	\caption{Forbidden configuration in \cref{obs_5v_inc_4_3f}.}
	\label{fig:obs_5v_inc_4_3f}
\end{figure}
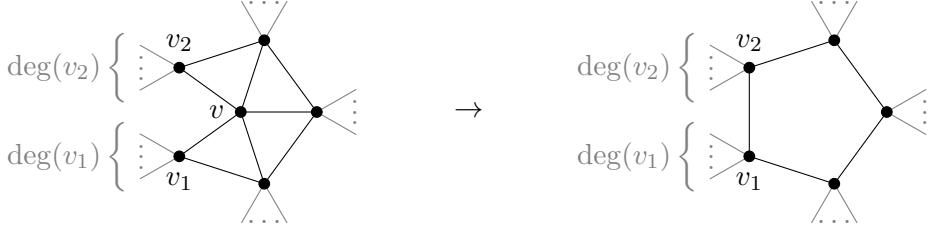

In the next claims we will need the following definitions.
For a vertex $v$, a vertex $u \in N(v)$ is a \emph{bad neighbor of $v$} if (i) $\deg(u)=5$ and (ii) $u$ is incident to four $3$-faces and one $4^+$-face $f$, and (iii) $v$ and $u$ are consecutive vertices of the boundary of $f$.
A vertex $u \in N(v)$ is \emph{very bad neighbor of $v$} if $\deg(u)=5$ and $u$ is incident to five $3$-faces.

\begin{claim}\label{5v_5f_adj}
The very bad neighbors of any vertex are pairwise nonadjacent.
\end{claim}
\begin{proof}
Assume that some vertex has two adjacent very bad neighbors $v_1,v_2$, see \cref{fig:5v_5f_adj}.
Let $G' = G - v_1v_2$.
By the minimality of $G$, we observe that $G'$ admits a distance-2 $(3\Delta + 4)$-coloring $\varphi$.
Finally, note that $\varphi$ is also a distance-2 coloring of $G$, a contradiction.
\end{proof}
\begin{figure}[h!]
		\begin{center}
		\begin{tikzpicture}
		\coordinate (v) at (0,0);
		\coordinate (u) at (1,0);
		
		\foreach \x in {1,2,3,4}
		{
			\coordinate (v\x) at ($(v)+(72*\x:1)$);
			\draw (v) -- (v\x);
		}
		\draw (v) -- (u) -- (v1) -- (v2) -- (v3) -- (v4) -- (u);
		\foreach \x in {1,2}
		{
			\coordinate (u\x) at ($(u)+(180-72*\x-72:1)$);
			\draw (u) -- (u\x);
		}
		\draw (v1) -- (u1) -- (u2) -- (v4);
		
		\leftstar{(v2)}{\Delta-3}
		\leftstar{(v3)}{\Delta-3}
		\upstar{(v1)}{\Delta-4}
		\downstar{(v4)}{\Delta-4}
		\rightstar{(u1)}{\Delta-3}
		\rightstar{(u2)}{\Delta-3}
		
		\filldraw[black] (v) circle (2pt) node[inner sep=5pt,anchor=east] {$v_1$};
		\filldraw[black] (u) circle (2pt) node[inner sep=6pt,anchor=west] {$v_2$};
		
		\foreach \x in {1,2,3,4}
		\filldraw[black] (v\x) circle (2pt);
		\foreach \x in {1,2}
		\filldraw[black] (u\x) circle (2pt);
		
		\node at (4,0) {$\to$};
		
		\coordinate (v) at (7,0);
		\coordinate (u) at (8,0);
		
		\foreach \x in {1,2,3,4}
		{
			\coordinate (v\x) at ($(v)+(72*\x:1)$);
			\draw (v) -- (v\x);
		}
		\draw (u) -- (v1) -- (v2) -- (v3) -- (v4) -- (u);
		\foreach \x in {1,2}
		{
			\coordinate (u\x) at ($(u)+(180-72*\x-72:1)$);
			\draw (u) -- (u\x);
		}
		\draw (v1) -- (u1) -- (u2) -- (v4);
		
		\leftstar{(v2)}{\Delta-3}
		\leftstar{(v3)}{\Delta-3}
		\upstar{(v1)}{\Delta-4}
		\downstar{(v4)}{\Delta-4}
		\rightstar{(u1)}{\Delta-3}
		\rightstar{(u2)}{\Delta-3}
		
		\filldraw[black] (v) circle (2pt) node[inner sep=5pt,anchor=east] {$v_1$};
		\filldraw[black] (u) circle (2pt) node[inner sep=6pt,anchor=west] {$v_2$};
		
		\foreach \x in {1,2,3,4}
		\filldraw[black] (v\x) circle (2pt);
		\foreach \x in {1,2}
		\filldraw[black] (u\x) circle (2pt);
		\end{tikzpicture}
	\end{center}
	\caption{Forbidden configuration in \cref{5v_5f_adj}.}
	\label{fig:5v_5f_adj}
\end{figure}
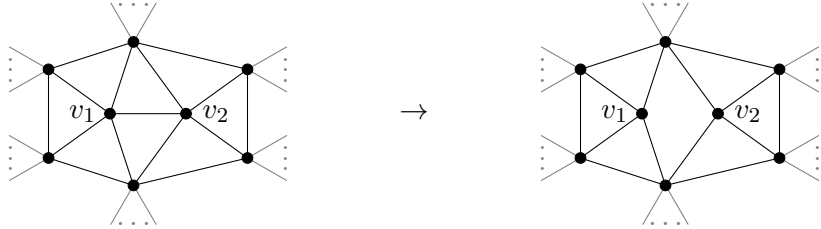

The proofs of Claims~\ref{5v_5f_4f_adj} and~\ref{5v_4f_adj} are again similar to each other.
In both of them the forbidden configuration involves two adjacent vertices $v_1$ and $v_2$.
We assume that such vertices exist in $G$ and obtain a new graph $G'$ by removing the edge $v_1v_2$.
Note that all pairs of vertices from $V(G) \setminus \{v_1,v_2\}$ that are at distance at most 2 in $G$ remain so in $G'$.
By the minimality of $G$, we observe that $G'$ admits a distance-2 $(3\Delta + 4)$-coloring.
However, after restoring the edge $v_1v_2$ the colors assigned to $v_1$ and $v_2$ might be in conflict with each other and with the colors of some other vertices. Thus we erase the colors of $v_1$ and $v_2$ are recolor these vertices in a greedy way.
Again, we ensure that this is possible by counting the number of blocked colors.

\begin{claim}\label{5v_5f_4f_adj}
If $v$ is an $8^-$-vertex, $v_1$ is its very bad neighbor and $v_2$ is its bad neighbor, then $v_1$ and $v_2$ are nonadjacent.
\end{claim}
\begin{proof}
For contradiction, suppose that $G$ has three vertices $v,v_1,v_2$ as in the assumption and $v_1$ is adjacent to $v_2$, see \cref{fig:5v_5f_4f_adj}.
Let $G' = G - v_1v_2$ and consider a distance-2 $(3\Delta + 4)$-coloring of $G'$ with colors of $v_1$ and $v_2$ erased.
We observe that the number of colors that are blocked for $v_1$ is at most $2(\Delta - 2) + (\Delta - 3) + 2 + 6 = 3\Delta + 1 < 3\Delta+4$. %
Then the number of colors that are blocked for $v_2$ is at most $(\Delta - 1) + (\Delta - 2) + (\Delta - 3)  + 2 + 6 = 3\Delta + 2 < 3\Delta+4$.
\end{proof}
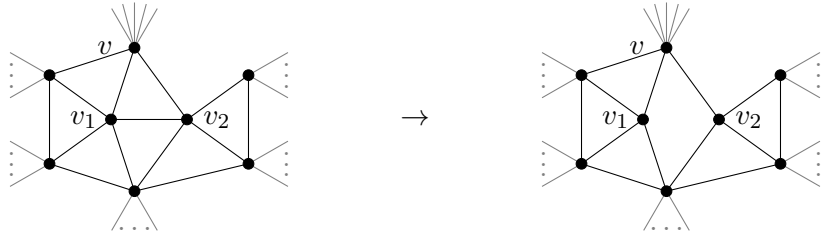
\begin{figure}[h!]
	\begin{center}
		\begin{tikzpicture}
		\coordinate (v) at (0,0);
		\coordinate (u) at (1,0);
		
		\foreach \x in {1,2,3,4}
		{
			\coordinate (v\x) at ($(v)+(72*\x:1)$);
			\draw (v) -- (v\x);
		}
		\draw (v) -- (u) -- (v1) -- (v2) -- (v3) -- (v4) -- (u);
		\foreach \x in {1,2}
		{
			\coordinate (u\x) at ($(u)+(180-72*\x-72:1)$);
			\draw (u) -- (u\x);
		}
		\draw (u1) -- (u2) -- (v4);
		
		\foreach \x in {0,1,2,3,4}
		\draw[gray] (v1) -- ($(v1)+(60+15*\x:0.6)$);
		
		\leftstar{(v2)}{\Delta-3}
		\leftstar{(v3)}{\Delta-3}
		\downstar{(v4)}{\Delta-4}
		\rightstar{(u1)}{\Delta-3}
		\rightstar{(u2)}{\Delta-3}
		
		\filldraw[black] (v) circle (2pt) node[inner sep=5pt,anchor=east] {$v_1$};
		\filldraw[black] (u) circle (2pt) node[inner sep=6pt,anchor=west] {$v_2$};
		
		\filldraw[black] (v1) circle (2pt) node[inner sep=8pt,anchor=east] {$v$};
		
		\foreach \x in {2,3,4}
		\filldraw[black] (v\x) circle (2pt);
		\foreach \x in {1,2}
		\filldraw[black] (u\x) circle (2pt);

		\node at (4,0) {$\to$};
		
		\coordinate (v) at (7,0);
		\coordinate (u) at (8,0);
		
		\foreach \x in {1,2,3,4}
		{
			\coordinate (v\x) at ($(v)+(72*\x:1)$);
			\draw (v) -- (v\x);
		}
		\draw (u) -- (v1) -- (v2) -- (v3) -- (v4) -- (u);
		\foreach \x in {1,2}
		{
			\coordinate (u\x) at ($(u)+(180-72*\x-72:1)$);
			\draw (u) -- (u\x);
		}
		\draw (u1) -- (u2) -- (v4);
		
		\foreach \x in {0,1,2,3,4}
		\draw[gray] (v1) -- ($(v1)+(60+15*\x:0.6)$);
		
		\leftstar{(v2)}{\Delta-3}
		\leftstar{(v3)}{\Delta-3}
		\downstar{(v4)}{\Delta-4}
		\rightstar{(u1)}{\Delta-3}
		\rightstar{(u2)}{\Delta-3}
		
		\filldraw[black] (v) circle (2pt) node[inner sep=5pt,anchor=east] {$v_1$};
		\filldraw[black] (u) circle (2pt) node[inner sep=6pt,anchor=west] {$v_2$};
		
		\filldraw[black] (v1) circle (2pt) node[inner sep=8pt,anchor=east] {$v$};
		
		\foreach \x in {2,3,4}
		\filldraw[black] (v\x) circle (2pt);
		\foreach \x in {1,2}
		\filldraw[black] (u\x) circle (2pt);
		\end{tikzpicture}
	\end{center}
	\caption{Forbidden configuration in \cref{5v_5f_4f_adj}.}	\label{fig:5v_5f_4f_adj}
\end{figure}

\begin{claim}\label{5v_4f_adj}
If $v$ is a $7$-vertex and $v_1,v_2$ are its bad neighbors, then $v_1$ and $v_2$ are nonadjacent.
\end{claim}
\begin{proof}
For contradiction, suppose that $G$ has vertices $v,v_1,v_2$ as in the assumption of the claim and $v_1$ is adjacent to $v_2$, see \cref{fig:5v_4f_adj}.
Let $G' = G - v_1v_2$ and consider a distance-2 $(3\Delta + 4)$-coloring of $G'$ with colors of $v_1$ and $v_2$ erased.
Note that the number of colors blocked for each of $v_1,v_2$ is at most $(\Delta - 1) + (\Delta - 2) + (\Delta - 3)  + 2 + 6 = 3\Delta + 2 < 3\Delta+4$.
\end{proof}
\begin{figure}[h!]	
	\begin{center}
		\begin{tikzpicture}
		\coordinate (v) at (0,0);
		\coordinate (u) at (1,0);
		
		\foreach \x in {1,2,3,4}
		{
			\coordinate (v\x) at ($(v)+(72*\x:1)$);
			\draw (v) -- (v\x);
		}
		\draw (v) -- (u) -- (v1);
		\draw (v2) -- (v3) -- (v4) -- (u);
		\foreach \x in {1,2}
		{
			\coordinate (u\x) at ($(u)+(180-72*\x-72:1)$);
			\draw (u) -- (u\x);
		}
		\draw (u1) -- (u2) -- (v4);
		
		\foreach \x in {0,1,2,3,4}
		\draw[gray] (v1) -- ($(v1)+(60+15*\x:0.6)$);
		
		\leftstar{(v2)}{\Delta-3}
		\leftstar{(v3)}{\Delta-3}
		\downstar{(v4)}{\Delta-4}
		\rightstar{(u1)}{\Delta-3}
		\rightstar{(u2)}{\Delta-3}
		
		\filldraw[black] (v) circle (2pt) node[inner sep=5pt,anchor=east] {$v_1$};
		\filldraw[black] (u) circle (2pt) node[inner sep=6pt,anchor=west] {$v_2$};
		
		\filldraw[black] (v1) circle (2pt) node[inner sep=8pt,anchor=east] {$v$};
		
		\foreach \x in {2,3,4}
		\filldraw[black] (v\x) circle (2pt);
		\foreach \x in {1,2}
		\filldraw[black] (u\x) circle (2pt);

		\node at (4,0) {$\to$};
		
		\coordinate (v) at (7,0);
		\coordinate (u) at (8,0);
		
		\foreach \x in {1,2,3,4}
		{
			\coordinate (v\x) at ($(v)+(72*\x:1)$);
			\draw (v) -- (v\x);
		}
		\draw (u) -- (v1);
		\draw (v2) -- (v3) -- (v4) -- (u);
		\foreach \x in {1,2}
		{
			\coordinate (u\x) at ($(u)+(180-72*\x-72:1)$);
			\draw (u) -- (u\x);
		}
		\draw (u1) -- (u2) -- (v4);
		
		\foreach \x in {0,1,2,3,4}
		\draw[gray] (v1) -- ($(v1)+(60+15*\x:0.6)$);
		
		\leftstar{(v2)}{\Delta-3}
		\leftstar{(v3)}{\Delta-3}
		\downstar{(v4)}{\Delta-4}
		\rightstar{(u1)}{\Delta-3}
		\rightstar{(u2)}{\Delta-3}
		
		\filldraw[black] (v) circle (2pt) node[inner sep=5pt,anchor=east] {$v_1$};
		\filldraw[black] (u) circle (2pt) node[inner sep=6pt,anchor=west] {$v_2$};
		
		\filldraw[black] (v1) circle (2pt) node[inner sep=8pt,anchor=east] {$v$};
		
		\foreach \x in {2,3,4}
		\filldraw[black] (v\x) circle (2pt);
		\foreach \x in {1,2}
		\filldraw[black] (u\x) circle (2pt);
		\end{tikzpicture}
	\end{center}
	\caption{Forbidden configiration in \cref{5v_4f_adj}.}
	\label{fig:5v_4f_adj}
\end{figure}

\subsection{Discharging}\label{sec:discharging}
We give an initial charge of $\wei(x) = \deg(x) - 4$ to every $x \in V(G) \cup F(G)$.
Using the Euler's formula $|V(G)| - |E(G)| + |F(G)| = 2$ and the handshaking lemma $2 |E(G)| = \sum_{v \in V(G)} \deg(v) =  \sum_{f \in F(G)} \deg(f)$, we derive the equality 
\begin{equation} \label{total_charge}
\sum_{x \in  V(G) \cup F(G)} \wei(x) = \sum_{x \in  V(G) \cup F(G)} (\deg(x) - 4) = -8. 
\end{equation}

The charges will be transferred between the elements of $V(G) \cup F(G)$ according to six discharging rules.
As we will see at the end, after the application of these rules, each vertex and face will have a non-negative charge.
Thus the total charge must be non-negative, which contradicts \eqref{total_charge}.
This shows that a counterexample to \cref{main_theorem} cannot exist.

\subsubsection*{Discharging rules}

We apply the following discharging rules.

\begin{enumerate}[label=\textbf{R\arabic*}]
\item \label{R1} Every $5^{+}$-vertex sends $\frac{1}{3}$ to each incident 3-face.
\item \label{R2} Every $6^{+}$-vertex sends additional $\frac{1}{6}$ to each incident weird 3-face.
\item \label{R3} Every $5^{+}$-face sends $\frac{1}{2}$ to each incident 3-vertex.
\item \label{R4} Every 6-vertex sends $\frac{1}{6}$ to each bad neighbor.
\item \label{R5} Every $7^{+}$-vertex sends $\frac{1}{3}$ to each bad neighbor.
\item \label{R6} Every $7^{+}$-vertex sends $\frac{1}{6}$ to each very bad neighbor.
\end{enumerate}

\subsubsection*{Final charges}

For $x \in V(G) \cup V(F)$, let $\wei'(x)$ be the final charge (after applying discharging rules).
We aim to show that $\wei'(x)\geq 0$ for all $x\in V(G) \cup F(G)$.

Let us start by bounding $\wei'(v)$ for $v \in V(G)$.
As by \cref{delta3} we have $\deg(v) \geq 3$. 
The analysis is split into cases depending on the degree of $v$.

\paragraph*{\textbf{Case $\deg(v)=3$.}}
By \cref{3v_inc_3f} and \cref{3v_inc_2_4f}, $v$ is incident to at least two $5^+$-faces.
Note that $v$ does not lose any charge and by \ref{R3} it receives charge at least $2\cdot \frac12=1$.
Consequently $\wei'(v)\geq \wei(v) +1 = (3-4)+1=0$.

\paragraph*{\textbf{Case $\deg(v)=4$.}}
As $v$ does not lose nor receive any charge, we have $\wei'(v) = \wei(v) = (4-4)=0$.

\paragraph*{\textbf{Case $\deg(v)=5$.}} 
Recall that $v$ loses charge due to \ref{R1} and might receive charge by \ref{R4}, \ref{R5}, and \ref{R6}.
Consider the following subcases.
\begin{itemize}
\item If $v$ is incident to at most three 3-faces, then the charge lost by $v$ due to \ref{R1} is at most $3\cdot\frac{1}{3}=1$.
Thus, $\wei'(v)\geq \wei(v) -1 = (5-4)-1=0$.
\item If $v$ is incident to four 3-faces, then due to \ref{R1} the charge lost is $4\cdot\frac{1}{3}=\frac{4}{3}$.
By \cref{obs_5v_inc_4_3f}, rules \ref{R4} and \ref{R5} make $v$ receive charge at least $\min\{2\cdot\frac{1}{6},\frac{1}{3}\}=\frac{1}{3}$.
Thus, $\wei'(v)\geq \wei(v) -\frac{4}{3}+\frac{1}{3} = (5-4)-1=0$.
\item If $v$ is incident to five 3-faces, then $v$ loses charge $5\cdot\frac{1}{3}=\frac{5}{3}$ by \ref{R1}.
By \cref{5v_inc_5_3f}, the application of \ref{R6} makes $v$ receive charge at least $4\cdot\frac{1}{6}=\frac{2}{3}$.
Thus, $\wei'(v)\geq \wei(v) -\frac{5}{3}+\frac{2}{3} =  (5-4)-1=0$.
\end{itemize}

\paragraph*{\textbf{Case $\deg(v)=6$.}} 
Recall that $v$ never receives any charge and loses charge due to \ref{R1}, \ref{R2}, and \ref{R4}.
Consider the following subcases.
\begin{itemize}
\item If $v$ is incident to six 3-faces, then $v$ has no bad neighbors.
Furthermore, by \cref{noweirdface}, $v$ is not incident to a weird 3-face.
Therefore $v$ loses charge only do due to \ref{R1} and the total charge lost is $6\cdot\frac{1}{3}=2$, so $\wei'(v)= \wei(v)-2 = (6-4)-2=0$.
\item Now consider the case that $v$ is incident to at most five 3-faces. Let $\mathbf{F}$ be the set of fans centered at $v$.
For each $\mathbf{f} \in \mathbf{F}$, there are $|\mathbf{f}|+1$ neighbors of $v$ that are incident to a face from $\mathbf{f}$.
Furthermore, each neighbor of $v$ is incident to a face of at most one fan.
Thus, $\sum_{\mathbf{f} \in \mathbf{F}}(|\mathbf{f}|+1)\leq \deg(v)=6$.

If a fan $\mathbf{f} \in \mathbf{F}$ contains a weird 3-face $f$, then, by \cref{weirdouter}, the face $f$ can only be an outer face of $\mathbf{f}$. Furthermore, the 4-vertex from the boundary of $f$ must be an outer vertex of $\mathbf{f}$.
Moreover, every bad neighbor of $v$ is an outer vertex of some fan in $\mathbf{F}$.
Thus the total number of bad neighbors and weird 3-faces incident to $v$ is bounded by $2 |\mathbf{F}|$.
Summing up, by \ref{R1}, \ref{R2}, and \ref{R4}, the total charge lost by $v$ is at most
\[
\frac{1}{3} \sum_{\mathbf{f} \in \mathbf{F}} |\mathbf{f}|+ \frac{1}{6} \cdot 2 \cdot |\mathbf{F}| = 
\frac{1}{3} \sum_{\mathbf{f} \in \mathbf{F}} |\mathbf{f}|+\frac{1}{3} \sum_{\mathbf{f} \in \mathbf{F}} 1 =\frac{1}{3} \sum_{\mathbf{f} \in \mathbf{F}}\left(|\mathbf{f}|+1\right)\leq\frac {\deg(v)}{3}=2.
\]
Consequently, $\wei'(v) \geq \wei(v) -2 =(6-4)-2=0$.
\end{itemize}

\paragraph*{\textbf{Case $\deg(v)=7$.}} 
Recall that $v$ never receives any charge and loses charge due to \ref{R1}, \ref{R2}, \ref{R5}, and \ref{R6}.
Consider the following subcases.
\begin{itemize}
\item  If $v$ is incident to seven 3-faces, then $v$ has no bad neighbor and, by \cref{5v_5f_adj}, $v$ has at most three very bad neighbors.
Furthermore, by \cref{noweirdface}, $v$ is not incident to a weird 3-face.
Thus, by \ref{R1} and \ref{R6}, the total charge lost by $v$ is at most $7\cdot\frac{1}{3} +3\cdot\frac{1}{6}=\frac{17}{6}$ and thus $\wei'(v) \geq \wei(v) - \frac{17}{6} = (7-4)-\frac{17}{6}=\frac{1}{6} \geq 0$.

\item Now consider the case that $v$ is incident to at most six 3-faces.
Let $\mathbf{F}$ be the set of fans centered at $v$. Recall that any two elements of $\mathbf{F}$ are separated by at least $4^+$-face incident to $v$. We proceed similarly as for the case of 6-vertices.

Consider a fan $\mathbf{f} \in \mathbf{F}$.
By \cref{weirdouter}, if $\mathbf{f}$ contains a weird 3-face $f$, then $f$  must be an outer face of $\mathbf{f}$ and the 4-vertex on the boundary of $f$ must be an outer vertex of $\mathbf{f}$.
By \cref{5v_5f_adj}, \cref{5v_5f_4f_adj} and \cref{5v_4f_adj}, bad and very bad neighbors of $v$ are pairwise nonadjacent.
Moreover, every bad neighbor which is a vertex of $\mathbf{f}$ must be an outer vertex of $\mathbf{f}$.

Let $w(\mathbf{f})$ be total charge sent by $v$ to the vertices and faces of $\mathbf{f}$.
The contribution of \ref{R1} to $w(\mathbf{f})$ is exactly $\frac{1}{3} \cdot |\mathbf{f}|$.

If $|\mathbf{f}|=1$, then the possible cases are as follows:
(a) the unique face in $\mathbf{f}$ is weird and none of the vertices of $\mathbf{f}$ is a bad or a very bad neighbor of $v$,
(b) the unique face in $\mathbf{f}$ is not weird, one of the vertices of $\mathbf{f}$ is a bad neighbor of $v$, and the other one is not bad nor very bad,
(c) the unique face in $\mathbf{f}$ is not weird, one of the vertices of $\mathbf{f}$ is a very bad neighbor of $v$, and the other one is not bad nor very bad,
(d) the unique face in $\mathbf{f}$ is not weird and none of the vertices of $\mathbf{f}$ is a bad or a very bad neighbor of $v$.
Summing up, the the total contribution of \ref{R2}, \ref{R5}, and \ref{R6} to $w(\mathbf{f})$ is at most $\frac{1}{3}$ (this happens in case (b) above).

Now consider the case that $|\mathbf{f}| > 1$.
The total charge sent to the outer faces of $\mathbf{f}$ and their vertices  is at most $\frac{2}{3}$; it happens if each of outer vertices is a bad neighbor.
Now let us consider the vertices $u$ of $\mathbf{f}$ that are not incident to the outer faces, note that there are exactly  $(|\mathbf{f}|+1)-4=|\mathbf{f}|-3$ of them. Due to \ref{R6}, the vertex $v$ sends charge $\frac{1}{6}$ to $u$ if $u$ is a very bad neighbor of $v$.
Since very bad neighbors are pairwise nonadjacent, we conclude that the number of very bad neighbors that are not the vertices of outer faces of $\mathbf{f}$ is at most $\lceil  \frac{|\mathbf{f}|-3}{2} \rceil$.

Summing up, we obtain that $w(\mathbf{f}) \leq \alpha(|\mathbf{f}|)$, where
\[
\alpha(t) = \begin{cases}
\frac{4}{6} & \text{ if } t=1,\\
\frac{8}{6} & \text{ if } t=2,\\
\frac{10}{6} & \text{ if } t=3,\\
\frac{13}{6} & \text{ if } t=4,\\
\frac{15}{6} & \text{ if } t=5,\\
\frac{18}{6} & \text{ if } t=6.
\end{cases}
\]

A straightforward case analysis shows that $\sum_{\mathbf{f} \in \mathbf{F}} \alpha(|\mathbf{f}|) \leq 3$.
Thus we obtain that $\wei'(v) = \wei(v) - \sum_{\mathbf{f} \in \mathbf{F}} w(\mathbf{f}) \geq (7-4) - \sum_{\mathbf{f} \in \mathbf{F}} \alpha(|\mathbf{f}|) \geq 3-3=0$.
\end{itemize}

\paragraph*{\textbf{Case $\deg(v) \geq 8$.}} 
Similarly to the previous case, $v$ never receives any charge and loses charge due to \ref{R1}, \ref{R2}, \ref{R5}, and \ref{R6}.
Consider the following subcases.
\begin{itemize}
\item If all faces incident to $v$ are 3-faces, then $v$ has no bad neighbor and thus \ref{R5} does not apply.
Recall that if $v$ is incident to a weird 3-face, then no vertex indicent to that face is a very bad neighbor of $v$.
So the total charge lost by $v$ due to \ref{R1}, \ref{R2}, and \ref{R6} is at most $\frac{1}{3} \cdot \deg(v) + \frac{1}{6} \cdot \deg(v) = \frac{1}{2} \cdot \deg(v)$. This gives us the final charge $\wei'(v) \geq \wei(v) - \frac{1}{2} \deg(v) = (\deg(v) -4) - \frac{1}{2} \deg(v) \geq 0$ as $\deg(v) \geq 8$.

\item Assume that $v$ is incident to some $4^+$ face.
Let $\mathbf{F}$ be the set of fans centered at $v$ and consider some $\mathbf{f} \in \mathbf{F}$.
Again, let $w(\mathbf{f})$ be the total charge sent by $v$ to vertices and faces of $\mathbf{f}$.
The contribution of \ref{R1} to  $w(\mathbf{f})$ is $\frac{1}{3}  |\mathbf{f}|$.
Denoting by $p$ the number of weird 3-faces in $\mathbf{f}$, we obtain the the contribution of \ref{R2} to  $w(\mathbf{f})$ is $\frac{1}{6} p$.
Recall that bad and very bad neighbors of $v$ are 5-vertices.
By \cref{4v_inc_3f_466}, the number of 5-vertices in $\mathbf{f}$ is at most $(|\mathbf{f}|+1)-(p+1)=|\mathbf{f}|-p$.
Recall that if $\mathbf{f}$ contains a bad neighbor $u$ of $v$, then $u$ must be an outer vertex of $\mathbf{f}$, so at most 2 of these 5-vertices are bad neighbors of $v$.
Thus the total contribution of \ref{R5} and \ref{R6} to $w(\mathbf{f})$ is at most $\frac{1}{3} \cdot 2 + \frac{1}{6} \cdot (|\mathbf{f}|-p-2)$.
Summing up, the obtain $w(\mathbf{f}) \leq \frac{1}{3}  |\mathbf{f}| + \frac{1}{6} p + \frac{1}{3} \cdot 2 + \frac{1}{6} \cdot (|\mathbf{f}|-p-2) = \frac{1}{2} |\mathbf{f}| + \frac{1}{3} \leq \frac{1}{2} (|\mathbf{f}|+1)$.
So the final charge is $\wei'(v) \geq \wei(v) - \sum_{\mathbf{f} \in \mathbf{F}} w(\mathbf{f}) \geq (\deg(v) - 4) - \sum_{\mathbf{f} \in \mathbf{F}} \frac{1}{2} (|\mathbf{f}|+1) = (\deg(v) - 4) - \frac{1}{2} \deg(v) \geq 0$ as $\deg(v) \geq 8$. 
\end{itemize}

Now let us consider the values of $\wei'(f)$ for $f \in F(G)$.
Again, we consider the cases.

\paragraph*{\textbf{Case $\deg(f) = 3$.}} Recall that $f$ receives charge due to \ref{R1} and possibly \ref{R2} and never loses any charge.
By \cref{3v_inc_3f} and \cref{4v_inc_3f_466}, $f$ is incident to three $5^+$-vertices or at least two $6^+$-vertices.
By \ref{R1} and \ref{R2}, $f$ receives charge at least $\min\{3\cdot \frac{1}{3},2\cdot\left(\frac{1}{3}+\frac{1}{6}\right)\}=1$,
thus the final charge is $\wei'(f) \geq \wei(f) + 1 = (3-4)+1 =0$.

\paragraph*{\textbf{Case $\deg(f) = 4$.}} Recall that $f$ never sends nor receives any charge, so $\wei'(f)=\wei(v) = 4-4=0$.

\paragraph*{\textbf{Case $\deg(f) \geq  5$.}} Recall that $f$ never receives any charge and loses charge due to \ref{R3}.
By \cref{adj_3v}, the number of 3-vertices incident to $f$ is at most $\lfloor \frac{\deg(f)}{2} \rfloor$.
So the total charge lost by $f$ is at most $\frac{1}{2} \cdot \lfloor \frac{\deg(f)}{2} \rfloor$ and the final charge is
\[
	\wei'(f)\geq \wei(f)-\frac{1}{2} \left\lfloor\frac{\deg(f)}{2}\right\rfloor\geq
	\begin{cases}
	(5-4)-\frac{1}{2}\cdot 2=0 & \text{ if } \deg(f)=5,\\
	(\deg(f)-4)-\frac{\deg(f)}4=\frac{3\deg(f)}4-4\geq\frac{18}4-4\geq0 & \text{ if } \deg(f)\geq6.
	\end{cases}
\]

Summing up, we showed that for every $x \in V(G) \cup E(G)$ it holds that $\wei'(x) \geq 0$.
As the application of discharging rules did not create any new charge, we have 
\[
0 \leq \sum_{x \in V(G) \cup F(G)} \wei'(x) = \sum_{x \in V(G) \cup F(G)} \wei(x),
\]
which contradicts \eqref{total_charge}. Thus the hypothetical counterexample to \cref{main_theorem} cannot exists. This completes the proof.

\paragraph{Acknowledgment.} The authors are sincerely grateful to Marthe Bonamy for introducing us to the problem.

\bibliographystyle{abbrv}
\bibliography{main}
\end{document}